\documentclass[11pt]{amsart}
\usepackage{amstext,amssymb,amsmath,amsbsy,dsfont,tikz}

 \usepackage[left=2.65cm,right=2.65cm,top=3.2cm,bottom=3.2cm]{geometry}

\usepackage{amsmath,amssymb,latexsym,dsfont}
\usepackage[small]{caption}
\usepackage{graphicx,color,mathrsfs,tikz}%wasysym,overpic,tikz,
\usepackage{subfigure,color}
\usepackage{cite}
\usepackage[colorlinks=true,urlcolor=blue,
citecolor=red,linkcolor=blue,linktocpage,pdfpagelabels,
bookmarksnumbered,bookmarksopen]{hyperref}
\usepackage[italian,english]{babel}
\usepackage{units}
\usepackage{enumitem}

\usepackage{hyperref}
\usepackage{cleveref}

\usepackage{tikz}
\usetikzlibrary{intersections}

\usepackage{amscd}
\usepackage{amsfonts}
\usepackage{indentfirst}
\usepackage{verbatim}
\usepackage{amsmath}
\usepackage{amsthm}
\usepackage{enumerate}
\usepackage{graphicx}
\usepackage{color}
\usepackage[OT1]{fontenc}
\usepackage[latin1]{inputenc}
\usepackage[english]{babel}
\usepackage{amssymb,esint}

\newtheorem{theorem}{Theorem}[section]
\newtheorem{lemma}{Lemma}[section]
\newtheorem{proposition}{Proposition}[section]

\newtheorem{remark}{Remark}[section]

\newtheorem{definition}{Definition}[section]
\usepackage{tikz}
\usetikzlibrary{intersections}
\newcommand{\supp}{\operatorname{supp}}

\newcommand{\eps}{\varepsilon}

\newcommand{\mR}{\mathbb{R}}

\newcommand{\mc}{\mathrm{c}}

\newcommand{\hxi}{\widetilde \xi}

\newcommand{\hrho}{\widetilde \rho}

\newcommand{\be}{\begin{equation}}
\newcommand{\ee}{\end{equation}}

\newcommand{\ba}{\begin{array}}
\newcommand{\ea}{\end{array}}

\newcommand{\bem}{\begin{multline}}
\newcommand{\eem}{\end{multline}}

\numberwithin{equation}{section}

\title[]{Exponential decay of solutions of damped wave equations in one dimensional space in  the $L^p$ framework for various boundary conditions}

\author[Y. Chitour]{Yacine Chitour}
        \address[Y. Chitour]{Laboratoire des signaux et syst\`emes, \newline\indent
        Universit\'e Paris Saclay, \newline\indent
        France}
\email{yacine.chitour@l2s.centralesupelc.fr }

\author[H.-M. Nguyen]{Hoai-Minh Nguyen}
\address[H.-M. Nguyen]{Laboratoire Jacques Louis Lions, \newline\indent
Sorbonne Universit\'e\newline\indent
Paris, France}
\email{hoai-minh.nguyen@sorbonne-universite.fr}

\begin{document}

\maketitle

\begin{abstract}
We establish the decay of the solutions of the damped wave equations in one dimensional space for the Dirichlet, Neumann, and dynamic boundary conditions where the damping coefficient is a function of space and time. The analysis is based on the study of the corresponding hyperbolic systems associated with the Riemann invariants. The key ingredient in the study of these systems is the use of the internal dissipation energy  to estimate  the difference of solutions with their mean values in an average sense.  
\end{abstract}

\tableofcontents

\section{Introduction} 
This paper is devoted to 
the decay of solution of the damped wave equations in one dimensional space in  the $L^p$-framework for $1 < p < + \infty$ for various boundary conditions where the damping depends  on space and time. More precisely, we consider the damped wave equation 
\be \label{Sys-W}
\left\{\ba{cl}
\partial_{tt} u - \partial_{xx} u + a \partial_t u = 0 & \mbox{ in } \mR_+ \times (0, 1), \\[6pt]
u(0, \cdot) = u_0, \quad \partial_t u (0, \cdot) = u_1 & \mbox{ on } (0, 1), 
\ea \right. 
\ee
equipped with one of the following boundary conditions: 
\be \label{Dirichlet}
\mbox{Dirichlet boundary condition: $u(t, 0) = u(t, 1) = 0, \mbox{ for } t\geq 0$}, 
\ee
\be\label{Neumann}
\mbox{Neumann boundary condition: $\partial_x u(t, 0) = \partial_x u(t, 1) = 0,   \mbox{ for } t \geq 0$}, 
\ee
and, for $\kappa > 0$, 
\be\label{Dynamic}
\mbox{dynamic boundary condition: $\partial_x u(t, 0) - \kappa \partial_t u(t, 0) = \partial_x u(t, 1) +  \kappa \partial_t u(t, 1)= 0,   \mbox{ for } t \geq 0$}.  
\ee
Here $u_0 \in W^{1, p}(0, 1)$ (with $u_0(0) = u_0(1) = 0$, i.e., $u_0 \in W^{1, p}_0(0, 1)$, in the case where the Dirichlet boundary condition is considered),  and  $u_1 \in L^p(0, 1)$ are the initial conditions. Moreover, $a \in L^\infty \big(\mR_+ \times (0, 1) \big)$ is assumed to verify the following hypothesis: 
\be\label{eq:hyp-a}
a\geq 0,\hbox{ and }\exists \lambda,\varepsilon_0>0,\ (x_0-\varepsilon_0,x_0+\varepsilon_0)\subset (0,1) \mbox{ such that }
a\geq \lambda\hbox{ on }\mathbb{R}_+\times (x_0-\varepsilon_0,x_0+\varepsilon_0),
\ee
i.e, $a$ is non-negative and $a (t, x) \ge \lambda > 0$ for $t \ge 0$ and for $x$ in some open subset of $(0, 1)$. The region where $a>0$ represents the region in which the damping term is active.

\medskip

The decay of the solutions of \eqref{Sys-W} equipped with either \eqref{Dirichlet}, or \eqref{Neumann}, or \eqref{Dynamic} has been extensively investigated  in the case where $a$ is independent of $t$, i.e., $a(t, x) = a(x)$ and mainly in the $L^2$-framework, i.e. within an Hilbertian setting. In this case, concerning the Dirichlet boundary condition, under the additional geometric multiplier condition on $a$, by the multiplier method, see e.g., \cite{Lions88-VolI,Komornik94}, one can prove that the solution decays exponentially, i.e., there exist positive constants $C$ and $\gamma$ independent of $u$ such that 
\be\label{Decay-W}
\| \partial_t u(t, \cdot) \|_{L^2(0, 1)} + \| \partial_x u(t, \cdot) \|_{L^2(0, 1)}  \le C e^{-\gamma t} \Big(\| \partial_t u(0, \cdot) \|_{L^2(0, 1)} + \| \partial_x u(0, \cdot) \|_{L^2(0, 1)} \Big),\ t\geq 0.
\ee
The assumption that $a$ satisfies the geometric multiplier condition is equivalent to the requirement that $a(x) \ge \lambda > 0$ on some neighbourhood of $0$ or $1$. Based on more sophisticate arguments in the seminal work of Bardos, Lebeau, and Rauch on the controllability of the wave equation \cite{BLR92}, Lebeau \cite{Lebeau96} showed that \eqref{Decay-W} also holds {\it without} the  geometric multiplier condition on $a$, see also the work of Rauch and Taylor \cite{RT74}.  
When the damping coefficient $a$ is also time-dependent, similar results have been obtained recently  by Le Rousseau et al. in \cite{LLTT17}.  It is worth noticing that strong stabilization, i.e., the energy decay to zero for each trajectory, has been established previously using  LaSalle's invariance argument  \cite{Dafermos78,Haraux78}.  The analysis of  the nonlinear setting associated with \eqref{Sys-W} can be found in \cite{HZ88, Zuazua90,Martinez99,MV00,CMP20} and the references therein. Similar results holds for the Neumann boundary condition \cite{Zuazua90, BLR92, Martinez99, LLTT17}.  Concerning the dynamic boundary condition without interior damping effect, i.e., $a \equiv 0$, the analysis for $L^2$-framework was previously initiated  by Quinn and Russell \cite{QR77}. They proved that  the energy exponentially decays in $L^2$-framework in one dimensional space. The exponential decay for higher dimensional space was proved by Lagnese  \cite{Lagnese83} using the multiplier technique (see also \cite{QR77}). The decay hence was established for the geometric multiplier condition and this technique was later extended in \cite{liu1997}, see also \cite{alabau2012} for a nice account on these issues.  

Much less is known about the asymptotic stability of  \eqref{Sys-W} equipped with either \eqref{Dirichlet}, or \eqref{Neumann}, or \eqref{Dynamic} in $L^p$-framework. This is probably due to the fact that for linear wave equations considered in domains of $\mR^d$ with $d \ge 2$ is not a well defined bounded operator in general in $L^p$ framework with $p  \neq 2$, a result due to Peral \cite{Peral80}.  As far as we know, the only work concerning 
exponential decay in the $L^p$-framework  is due to Kafnemer et al. \cite{KMJC22},  where the Dirichlet boundary condition was considered. For the damping coefficient $a$ being time-independent, they showed that the decay holds under the additional geometric multiplier condition on $a$ for $1 < p < + \infty$. Their analysis is via the multiplier technique involving  various non-linear test functions. In the case of zero damping and with a dynamic boundary condition, previous results have been obtained in \cite{CMM2021} where the problem has been reduced to the study of a discrete time dynamical system over appropriate functional spaces.

\medskip 
The goal of this paper is to give a unified approach to  deal with all the boundary considered in \eqref{Dirichlet}, \eqref{Neumann},  and \eqref{Dynamic} in the $L^p$-framework for $1 < p < + \infty$ under the condition \eqref{eq:hyp-a}. Our results thus hold even in the case where $a$ is a function of time and space. The analysis is based on the study of the corresponding hyperbolic systems associated with the Riemann invariants for which new insights are required. 

\medskip 
Concerning the Dirichlet boundary condition, we obtain the following result.

\begin{theorem}\label{thm-W} Let $1< p < + \infty$, $\eps_0 > 0$, $\lambda > 0$,  and let $a \in L^\infty \big(\mR_+ \times (0, 1) \big)$ be such that $a \ge 0$ and $a \ge \lambda > 0$ in $\mR_+ \times (x_0 - \eps_0, x_0 + \eps_0) \subset \mR_+ \times (0, 1)$ for some $x_0 \in (0, 1)$. Then there exist positive constants $C$ and $\gamma$ depending only on $p$, $\| a\|_{L^\infty\big(\mR_+ \times (0, 1) \big)}$, $\eps_0$, and $\lambda$  such that for all $u_0 \in W^{1, p}_0(0, 1)$ and $u_1 \in L^p(0, 1)$,  the unique weak solution $u \in C([0, + \infty); W^{1, p}_0(0, 1)) \cap C^1([0, + \infty); L^p(0, 1))$ of \eqref{Sys-W} and \eqref{Dirichlet} satisfies 
 \be\label{Decay-W-p}
\| \partial_t u(t, \cdot) \|_{L^p(0, 1)} + \| \partial_x u(t, \cdot) \|_{L^p(0, 1)} \le C e^{-\gamma t} \Big(\| u_1 \|_{L^p(0, 1)} + \| \partial_x u_0 \|_{L^p(0, 1)} \Big),\ t\geq 0.
\ee
\end{theorem}

The meaning of the (weak) solutions  given \Cref{thm-W} is given in \Cref{sect-WP} (see \Cref{def-WS}) and their well-posedness is also established there (see \Cref{pro-WP}). Our analysis is via the study of the decay of solutions of hyperbolic systems which are associated with \eqref{Sys-W} via the Riemann invariants. Such a decay for the hyperbolic system, even in the case $p=2$, is new to our knowledge. The analysis of these systems has its own interest and is  motivated by recent analysis on the controllability of hyperbolic systems in one dimensional space \cite{CoronNg15,CoronNg19,CoronNg19-2,CoronNg20}.

As in \cite{Haraux09, KMJC22},  we set 
\be \label{def-rhoxi}
\rho(t, x) = u_x(t, x) + u_t(t, x) \quad \mbox{ and } \quad \xi(t, x) = u_x(t, x) - u_t(t, x) \mbox{ for } (t, x) \in \mR_+ \times (0, 1). 
\ee 
One can check that for a smooth solution $u$ of \eqref{Sys-W} and \eqref{Dirichlet}, the pair of functions $(\rho, \xi)$ defined in \eqref{def-rhoxi} satisfies the system 
\be \label{Sys-H}
\left\{\ba{cl}
\rho_t  - \rho_x  = - \frac{1}{2} a (\rho - \xi) & \mbox{ in } \mR_+ \times (0, 1), \\[6pt]
\xi_t + \xi_x = \frac{1}{2} a  (\rho - \xi) & \mbox{ in } \mR_+ \times (0, 1), \\[6pt]
\rho(t, 0) - \xi (t, 0) = \rho(t, 1) - \xi(t, 1) = 0 & \mbox{ in } \mR_+. 
\ea\right.
\ee

One {\it cannot} hope the decay of a general solutions of \eqref{Sys-H} since any pair $(c, c)$ where $c \in \mR$ is a constant is a solution of \eqref{Sys-H}. Nevertheless, for $(\rho, \xi)$ being defined by \eqref{Sys-H} for a solution  $u$  of \eqref{Sys-W}, one also has the following additional information 
\be \label{cond-I}
\int_0^1 \rho(t, x) + \xi(t, x) \, d x = 0 \mbox{ for } t \ge 0. 
\ee

Concerning System \eqref{Sys-H} itself (i.e., without necessarily assuming \eqref{cond-I}), we prove the following result, which takes into account  \eqref{cond-I}. 

\begin{theorem}\label{thm-H} Let $1  <  p < + \infty$, $\eps_0 > 0$, $\lambda > 0$,  and $a \in L^\infty \big(\mR_+ \times (0, 1) \big)$ be such that $a \ge 0$ and $a \ge \lambda > 0$ in $\mR_+ \times (x_0 - \eps_0, x_0 + \eps_0) \subset \mR_+ \times (0, 1)$ for some $x_0 \in (0, 1)$.  There exist a positive constant $C$ and a positive constant $\gamma$ depending only on on $p$, $\| a\|_{L^\infty\big(\mR_+ \times (0, 1) \big)}$, $\eps_0$, and $\lambda$ such that the unique solution $(\rho,\xi)$ of \eqref{Sys-H} with the initial condition $\rho(0, \cdot) = \rho_0$ and $\xi(0, \cdot) = \xi_0$ satisfies 
\be \label{thmH-conclusion}
\| (\rho-c_0, \xi-c_0) (t, \cdot) \|_{L^p(0, 1)} \le C e^{-\gamma t} \|(\rho(0, \cdot) - c_0, \xi (0, \cdot)-c_0)\|_{L^p(0, 1)},\ t\geq 0, 
\ee
where 
\be\label{eq:c}
c_0:=\frac12 \int_0^1 \big(\rho(0, x) + \xi(0, x) \big) \, dx,  
\ee
\end{theorem}

In \Cref{thm-H},   we consider the broad solutions. It is understood through the broad solution in finite time: for $T>0$ and $1\le p < + \infty$, a broad solution $u$ of the system 
\be \label{Sys-H-1}
\left\{\ba{cl}
\rho_t - \rho_x = - \frac{1}{2} a (\rho - \xi) & \mbox{ in } (0, T) \times (0, 1), \\[6pt]
\xi_t + \xi_x = \frac{1}{2} a (\rho - \xi) & \mbox{ in } (0, T) \times (0, 1), \\[6pt]
\rho(t, 0) - \xi (t, 0) = \rho(t, 1) - \xi(t, 1) = 0  & \mbox{ in } (0, T), \\[6pt]
\rho(0, \cdot) = \rho_0, \quad \xi(0, \cdot) = \xi_0 &  \mbox{ in } (0, 1), 
\ea\right.
\ee
is a pair of functions $(\rho, \xi) \in C([0, T]; \big[L^p(0, 1) \big]^2 \big) \cap C([0, 1]; \big[L^p(0, T) \big]^2 \big)$ which obey the characteristic rules, see e.g., \cite{CoronNg19}. The well-posedness of \eqref{Sys-H-1} can be found in \cite{CoronNg19} (see also the appendix of \cite{CoronNg-T-21}). The analysis there is mainly for the case $p=2$ but the arguments extend naturally for the case $1 \le  p < + \infty$.

In the $L^p$-framework,  the Neumann boundary condition and its corresponding hyperbolic systems  are discussed in \Cref{sect-Neumann} and  the dynamic boundary  condition and its corresponding hyperbolic systems are discussed in \Cref{sect-Dynamic}.  Concerning the dynamic boundary condition, the decay holds even under the assumption $a \ge 0$. The analysis for the Neumann  case shares a large part in common with the one of the Dirichlet boundary condition. The difference in their analysis comes from taking into account differently the boundary condition. The analysis of the dynamic condition is similar but much simpler.

The study of the wave equation in one dimensional space via the corresponding hyperbolic system is known. The controllability and stability of hyperbolic systems has been also investigated extensively. This goes back to the work of Russel \cite{Russell73, Russell78} and Rauch and Taylor \cite{RT74}. Many important progress has been obtained recently, see,  e.g.,  \cite{2016-Bastin-Coron-book} and the references therein. It is worth noting that many works have been devoted to the $L^2$-framework.  Less is studied in the $L^p$-scale. In this direction, we want to mention \cite{CoronNg15} where the exponential stability is studied for dissipative boundary condition. 

Concerning the wave equation in one dimensional space, the exponential decay in $L^2$-setting for  the dynamic boundary condition is also established via its corresponding hyperbolic systems \cite{QR77}. However, to our knowledge, the exponential decay for the Dirichlet and Neumann boundary conditions has not been established  even in $L^2$-framework via this approach.  Our work is new and quite distinct from the one in \cite{QR77} and has its own interest. First, the analysis in \cite{QR77} uses essentially the fact that the boundary is strictly dissipative, i.e., $\kappa > 0$ in \eqref{Dynamic}. Thus the analysis cannot be used for the Dirichlet and Neumann boundary conditions. Moreover, it is not clear how to extend it to the $L^p$-framework. Concerning our analysis, the key observation is that the information of the internal energy  allows one to control  the difference of the solutions and its mean value in the interval of time $(0, T)$ in an average sense. This observation is implemented in  two lemmas (\Cref{lemH} and \Cref{lemM}) after using a standard result (\Cref{lem-Ineq}) presented in 
\Cref{sect-Lemmas}.  These two lemmas are the main ingredients of our analysis for the Dirichlet and Neumann boundary conditions. The proof of the first lemma is mainly based on the characteristic method while as the proof of the second lemma is inspired from the theory of functions with bounded mean oscillations due to John and Nirenberg \cite{JN61}. As seen later that, the analysis for the dynamic boundary condition is much simpler for which the use of  \Cref{lem-Ineq} is sufficient.  

An interesting point of our analysis is the fact that these lemmas do not depend on the boundary conditions used. In fact, one can apply it in a setting where a bound of the internal energy is accessible. This allows us to deal with all the boundary conditions considered in this paper by the same way. Another point of our analysis which is helpful to be mentioned is that  the asymptotic stability for hyperbolic systems in one dimensional space has been mainly studied for general solutions. This is not the case in the setting of \Cref{thm-H} where the asymptotic stability holds under  condition \eqref{cond-I}. It is also worth noting that  the time-dependent coefficients generally make the phenomena more complex, see \cite{CoronNg-T-21} for a discussion on the optimal null-controllable time.

The analysis in this paper cannot handle the cases $p=1$ and  $p= + \infty$. Partial results in this direction for the Dirichlet boundary condition can be found in  \cite{KMJC22} where $a$ is constant and in some range. These cases will be considered elsewhere by different approaches.

The paper is organized as follows. The well-posedness of \eqref{Sys-W} equipped with one of the boundary conditions  \eqref{Dirichlet} and  \eqref{Neumann} is discussed in \Cref{sect-WP}, where a slightly more general context is considered (the boundary condition \eqref{Dynamic} is considered directly in \Cref{sect-Dynamic};  comments on this point is given in  \Cref{rem-Dynamic}). \Cref{sect-Dirichlet} is devoted to the proof of \Cref{thm-W} and \Cref{thm-H}. We also relaxed slightly the non-negative assumption on $a$ in \Cref{thm-W} and \Cref{thm-H} there (see \Cref{thm-Hp} and \Cref{thm-Wp}) using a standard perturbative argument. The Neumann boundary condition is studied in \Cref{sect-Neumann} and  the Dynamic boundary condition is considered in \Cref{sect-Dynamic}.  

%We finally give a summary of the paper in \Cref{sect-Summary}

\section{The well-posedness in $L^p$-setting} \label{sect-WP}

In this section, we give the meaning of the solutions of the damped wave equation \eqref{Sys-W} equipped with either the Dirichlet boundary condition \eqref{Dirichlet} or the Neumann boundary condition \eqref{Neumann} and establish their well-posedness in the $L^p$-framework with $1 \le p \le + \infty$.  We will consider a slightly more general context. More precisely, we consider the system 
\be \label{Sys-WG}
\left\{\ba{cl}
\partial_{tt} u - \partial_{xx} u + a \partial_t u + b  \partial_x u + c  u  = f & \mbox{ in } (0, T) \times (0, 1), \\[6pt]
u(0, \cdot) = u_0, \quad \partial_t u (0, \cdot) = u_1 & \mbox{ in } (0, 1), 
\ea \right. 
\ee
equipped with either 
\be \label{DirichletT}
\mbox{Dirichlet boundary condition: $u(t, 0) = u(t, 1) = 0   \mbox{ for } t \in (0, T)$}, 
\ee
or  
\be\label{NeumannT}
\mbox{Neumann boundary condition: $\partial_x u(t, 0)  = \partial_x u(t, 1) = 0   \mbox{ for } t \in (0, T)$}.  
\ee
Here $a, b, c \in L^\infty((0, T) \times (0, 1))$ and $f \in L^p((0, T) \times (0, 1))$. 

\medskip 
We begin with the Dirichlet boundary condition.

\begin{definition}\label{def-WS} Let $T> 0$, $1 \le  p <  + \infty$, $a, b, c \in L^\infty((0, T) \times (0, 1))$, $f \in L^p((0, T) \times (0, 1))$, $u_0 \in W^{1, p}_0(0, 1)$, and $u_1 \in L^p(0, 1)$.  A function $u \in C([0, T]; W^{1, p}_0(0, 1)) \cap C^1([0, T]; L^p(0, 1))$ is called a (weak) solution of \eqref{Sys-WG} and \eqref{DirichletT} (up to time $T$)  if 
\be\label{def-WS-initial}
u(0, \cdot) = u_0,  \quad \partial_t u(0, \cdot) = u_1 \mbox{ in } (0, 1),  
\ee
and 
\begin{multline}\label{def-WS-identity}
\frac{d^2}{dt^2}\int_{0}^1 u (t, x) v(x) \, dx + \int_{0}^1 u_x(t, x) v_x(x) \, dx  
+ \int_0^1 a(t, x) u_t (t, x) v(x) \, dx \\[6pt] 
+ \int_0^1 b(t, x) u_x (t, x) v(x) \, dx  + \int_0^1 c(t, x) u (t, x) v(x) \, dx   =  \int_{0}^1 f(t, x) v(x) \, dx
\end{multline}
in the distributional sense in $(0, T)$ for all $v \in C^1_{c}(0, 1)$. 
\end{definition}

\Cref{def-WS} can be modified to deal with the case $p= +\infty$ as follows. 

\begin{definition}\label{def-WS-infty} Let $T> 0$,  $a, b, c \in L^\infty((0, T) \times (0, 1))$, $f \in L^\infty((0, T) \times (0, 1))$, $u_0 \in W^{1, \infty}_0(0, 1)$, and $u_1 \in L^\infty(0, 1)$. A function  $u \in L^\infty([0, T]; W^{1, \infty}_0(0, 1)) \cap W^{1, \infty}([0, T]; L^\infty(0, 1))$ is  called a (weak) solution of \eqref{Sys-WG} and \eqref{DirichletT} (up to time $T$) if $u \in C([0, T]; W^{1, 2}_0(0, 1)) \cap C^1([0, T]; L^2(0, 1))$ \footnote{By interpolation, one can use  $C([0, T]; W^{1, 2}_0(0, 1)) \cap C^1([0, T]; L^2(0, 1))$ instead of $C([0, T]; W^{1, 2}_0(0, 1)) \cap C^1([0, T]; L^2(0, 1))$ for any $1 \le q < + \infty$. This condition is used to give the meaning of the initial conditions.} and satisfies \eqref{def-WS-initial} and \eqref{def-WS-identity}. 
\end{definition}

Concerning the Neumann boundary condition, we have the following definition.  
\begin{definition}\label{def-WSN} Let $T> 0$, $1 \le  p <  + \infty$,  $a, b, c \in L^\infty((0, T) \times (0, 1))$, $f \in L^p((0, T) \times (0, 1))$, $u_0 \in W^{1, p}(0, 1)$, and $u_1 \in L^p(0, 1)$.  A function $u \in C([0, T]; W^{1, p}(0, 1)) \cap C^1([0, T]; L^p(0, 1))$ is called a (weak) solution of \eqref{Sys-WG} and \eqref{NeumannT} (up to time $T$)
if \eqref{def-WS-initial} is valid and 
\begin{multline}\label{def-WSNR-identity}
\frac{d^2}{dt^2}\int_{0}^1 u (t, x) v(x) \, dx + \int_{0}^1 u_x(t, x) v_x(x) \, dx    \\[6pt] 
+ \int_0^1 b(t, x) u_x (t, x) v(x) \, dx  + \int_0^1 c(t, x) u (t, x) v(x) \, dx 
+ \int_0^1 a(t, x) u_t (t, x) v(x) \, dx   =  \int_{0}^1 f(t, x) v(x) \, dx
\end{multline}
holds  in the distributional sense in $(0, T)$ for all $v \in C^1([0, 1])$. 
\end{definition}

\Cref{def-WS} can be modified to deal with the case $p= +\infty$ as follows. 

\begin{definition}\label{def-WSN-infty} Let $T> 0$,   $a, b, c \in L^\infty((0, T) \times (0, 1))$, $f \in L^\infty((0, T) \times (0, 1))$, $u_0 \in W^{1, \infty}(0, 1)$, and $u_1 \in L^\infty(0, 1)$. A function  $u \in L^\infty([0, T]; W^{1, \infty}(0, 1)) \cap W^{1, \infty}([0, T]; L^\infty(0, 1))$ is  called a (weak) solution of \eqref{Sys-WG} and \eqref{NeumannT} (up to time $T$) if $u \in C([0, T]; W^{1, 2}(0, 1)) \cap C^1([0, T]; L^2(0, 1))$ \footnote{By interpolation, one can use  $C([0, T]; W^{1, 2}_0(0, 1)) \cap C^1([0, T]; L^2(0, 1))$ instead of $C([0, T]; W^{1, 2}(0, 1)) \cap C^1([0, T]; L^2(0, 1))$ for any $1 \le q < + \infty$. This condition is used to give the meaning of the initial conditions.}, 
 \eqref{def-WS-initial} is valid, and \eqref{def-WS-identity} holds  in the distributional sense in $(0, T)$ for all $v \in C^1([0, 1])$. 
\end{definition}

Concerning the well-posedness of the Dirichlet system \eqref{Sys-WG} and \eqref{DirichletT}, we establish the following result. 

\begin{proposition}\label{pro-WP} Let $T> 0$, $1 \le  p \le   + \infty$, and $a, b, c \in L^\infty((0, T) \times (0, 1))$, and let $u_0 \in W^{1, p}_0(0, 1)$, $u_1 \in L^p(0, 1)$, and  $f \in L^p\big((0, T) \times (0, 1) \big)$.  Then there exists a unique (weak) solution $u$  of \eqref{Sys-WG} and \eqref{DirichletT}. Moreover, it holds  
\be  \label{pro-WP-statement}
\| \partial_t u(t, \cdot) \|_{L^p(0, 1)} + \| \partial_x u(t, \cdot) \|_{L^p(0, 1)}  \le C \Big(\|u_1 \|_{L^p(0, 1)} + \| \partial_x u_0 \|_{L^p(0, 1)} + \| f\|_{L^p\big((0, T) \times (0, 1)  \big)} \Big),\ t\geq 0
\ee
for some positive constant $C= C(p, T, \|a\|_{L^\infty}, \|b\|_{L^\infty}, \|c\|_{L^\infty})$ which is independent of $u_0$, $u_1$, and $f$. 
\end{proposition}

Concerning the well-posedness of the Neumann system \eqref{Sys-WG} and \eqref{NeumannT}, we prove the following result. 
 
\begin{proposition}\label{pro-WPN} Let $T> 0$, $1 \le  p \le   + \infty$, and $a, b, c \in L^\infty((0, T) \times (0, 1))$, and let $u_0 \in W^{1, p}(0, 1)$, $u_1 \in L^p(0, 1)$, and  $f \in L^p\big((0, T) \times (0, 1) \big)$.  Then there exists a unique (weak) solution $u$  of \eqref{Sys-WG} and \eqref{NeumannT} and 
\be\label{pro-WP-statement}
\| \partial_t u(t, \cdot) \|_{L^p(0, 1)} + \| \partial_x u(t, \cdot) \|_{L^p(0, 1)}  \le C \Big(\|u_1 \|_{L^p(0, 1)} + \| \partial_x u_0 \|_{L^p(0, 1)} + \| f\|_{L^p\big((0, T) \times (0, 1)  \big)} \Big),\ t\geq 0
\ee
for some positive constant $C= C(p, T, \|a\|_{L^\infty}, \|b\|_{L^\infty}, \|c\|_{L^\infty})$ which is independent of $u_0$, $u_1$, and $f$. 
\end{proposition}

\begin{remark} \rm 
The definition of weak solutions and the well-posedness are stated for $p=1$ and $p=+\infty$ as well. The existence and the well-posedness is well-known in the case $p=2$. The standard analysis in the case $p=2$ is via the Galerkin method. 
\end{remark}

The rest of this section is devoted to the proof of \Cref{pro-WP} and \Cref{pro-WPN} in \Cref{sect-proWP} and \Cref{sect-proWPN}, respectively. 

\subsection{Proof of \Cref{pro-WP}} \label{sect-proWP} The proof is divided into two steps in which we prove the uniqueness and the existence. 

$\bullet$ Step 1: Proof of the uniqueness.  Assume that $u$ is a (weak) solution of \eqref{Sys-WG} with $f = 0$ in $(0, T) \times (0, 1)$ and $u_0 = u_1 = 0$ in $(0, 1)$. 
We will show that $u  = 0$ in $(0, T) \times (0, 1)$. Set 
\be\label{pro-WP-def-g}
g(t, x) = - a(t, x) \partial_t u(t, x) - b(t, x) \partial_x u (t, x) - c(t, x) u(t, x). 
\ee
Then $u$ is a weak solution of the system 
\be \label{pro-WP-Sys1}
\left\{\ba{cl}
\partial_{tt} u - \partial_{xx} u  = g & \mbox{ in } (0, T) \times (0, 1), \\[6pt]
u(t, 0) = u(t, 1) = 0 & \mbox{ for } t \in (0, T), \\[6pt]
u(0, \cdot) = 0, \quad \partial_t u (0, \cdot) = 0 & \mbox{ in } (0, 1). 
\ea \right. 
\ee
Extend $u$ and $g$ in $(0, T) \times \mR$ by appropriate reflection  in $x$ first by odd extension in $(-1, 0)$, i.e., $u(t, x) = -  u (t, -x) $ and $g(t, x) = - g(t, -x)$ in $(0, T) \times (-1, 0)$
and so on,  and still denote the extension by $u$ and $g$. Then 
$u \in C([0, T]; W^{1, p}(-k, k)) \cap C^1([0, t]; L^p(-k, k))$ and $g \in L^p\big((0, T) \times (-k, k) \big)$ for $k \ge 1$ and for $1 \le p < + \infty$, and similar facts holds for $p=+ \infty$. We also obtain that  $u(0, \cdot) = 0$ and  $\partial_t u (0, \cdot) = 0$ in $\mR$, and   
\be \label{pro-WP-Sys2}
\partial_{tt} u - \partial_{xx} u  = g  \mbox{ in } (0, T) \times \mR \mbox{ in the distributional sense}. 
\ee
The d'Alembert formula gives, for $t\geq 0$, that 
\be\label{pro-WP-p0}
u(t, x) = \frac{1}{2}\int_0^t \int_{x - t + \tau}^{x + t - \tau} g(\tau, y)  \, dy \, d \tau. 
\ee
We then obtain  for $t\geq 0$
\be\label{pro-WP-p1}
\partial_t u(t, x) =  \frac{1}{2}\int_0^t   g(\tau, x + t - \tau) + g(\tau, x-t+ \tau)
\, d \tau
\ee
and 
\be\label{pro-WP-p2}
\partial_x u(t, x) =  \frac{1}{2}\int_0^t g(\tau, x + t -  \tau) - g(\tau, x - t + \tau)
\, d \tau. 
\ee
Using \eqref{pro-WP-def-g}, we  derive from \eqref{pro-WP-p0}, \eqref{pro-WP-p1} and \eqref{pro-WP-p2} that, for $1 \le p < + \infty$ and  for $t\geq 0$, 
\begin{multline} \label{pro-WP-p3}
\int_0^1 |\partial_t u(t, x)|^p + |\partial_t u(t, x)|^p + |\partial_x u(t, x)|^p \, dx  \\[6pt]
\le C   \int_0^t \int_0^1 \Big( |\partial_t u(s, y)|^p + |\partial_x u(s, y)|^p + |u(s, y)|^p \Big) \, d y \, ds, 
\end{multline}
and, for $p = + \infty$, 
\begin{multline}\label{pro-WP-p4}
\|u(t, \cdot)\|_{L^\infty(0, 1)} + \|\partial_t u(t, \cdot)\|_{L^\infty(0, 1)} + \|\partial_x u(t, \cdot) \|_{L^\infty(0, 1)} \\[6pt]
\le C t \Big( \|\partial_t u(t, \cdot)\|_{L^\infty\big((0,t) \times (0, 1)\big)} + \|\partial_x u(t, \cdot) \|_{L^\infty\big((0,t) \times (0, 1)\big)} +  \|u(t, \cdot) \|_{L^\infty\big((0,t) \times (0, 1)\big)} \Big),
\end{multline}
for positive constant $C$ only depending only on $p, T, \|a\|_{L^\infty}, \|b\|_{L^\infty}, \|c\|_{L^\infty}$. In the sequel, such constants will again be denoted by $C$. 

It is immediate to deduce from the above equations that $u= 0$ on $[0,1/2C] \times (0, 1)$ and then $u = 0 \mbox{ in } (0, T) \times (0, 1)$. 
The proof of the uniqueness is complete. 

\medskip 
$\bullet$ Step 2: Proof of the existence.  Let $(a_n)$, $(b_n)$, and $(c_n)$ be smooth functions in $[0, T] \times [0, 1]$ such that $\supp a_n, \supp b_n, \supp c_n \cap {0} \times [0, 1] = \emptyset$, 
$$
(a_n, b_n, c_n)  \rightharpoonup (a, b, c) \mbox{ weakly star in } \Big( L^\infty\big((0, T) \times (0, 1) \big) \Big)^3, 
$$
and  
$$
(a_n, b_n, c_n)  \to (a, b, c) \mbox{ in } \Big( L^q\big((0, T) \times (0, 1) \big) \Big)^3 \mbox{ for } 1 \le q < + \infty.  
$$

Let $u_{0, n} \in C^\infty_{\mc} (0, 1)$ and $u_{1, n} \in C^\infty_c(0, 1)$ be such that, if $1 \le p < + \infty$,  
$$
u_{0, n} \to u_0 \mbox{ in } W^{1, p}_0(0, 1) \quad \mbox{ and } \quad  u_{1, n} \to u_1 \mbox{ in } L^p(0, 1), 
$$
and, if $p = + \infty$ then the following two facts hold 
$$
u_{0, n} \rightharpoonup u_0 \mbox{ weakly star in  } W^{1, \infty}_0(0, 1) \quad \mbox{ and } \quad  u_{1, n} \rightharpoonup u_1 \mbox{ weakly star in } L^\infty(0, 1), 
$$
and, for $1 \le q < + \infty$,  
$$
u_{0, n} \to u_0 \mbox{ in } W^{1, q}_0(0, 1) \quad \mbox{ and } \quad  u_{1, n} \to u_1 \mbox{ in } L^q(0, 1). 
$$
The existence of $(a_n, b_n, c_n)$ and the existence of $u_{0, n}$ and $u_{1, n}$ follows from the standard theory of Sobolev spaces, see,  e.g., \cite{Brezis-FA}. 

 Let $u_n$ be the weak solution corresponding to $(a_n, b_n, c_n)$ with initial data $(u_{0, n}, u_{1, n})$. Then $u_n$ is smooth in $[0, T] \times [0, 1]$. Set 
$$
g_n(t, x) =  - a_n(t, x) \partial_t u_n(t, x) - b_n(t, x) \partial_x u_n (t, x) - c_n(t, x) u_n(t, x) \mbox{ in } (0, T) \times (0, 1).
$$
Extend $u_n$, $g_n$, and $f$ in $(0, T) \times \mR$ by first odd refection in $(-1, 0)$ and so on,  and still denote the extension by $u_n$ and $g_n$, and $f$. We then have 
\be \label{pro-WP-Sys3}
\partial_{tt} u_n - \partial_{xx} u_n  = g_n + f \mbox{ in } (0, T) \times \mR, 
\ee
The d'Alembert formula gives 
\begin{multline*}
u(t, x) = \frac{1}{2}\int_0^t \int_{x - t + \tau}^{x + t - \tau} g_n(\tau, y) + f (\tau, y) \, dy \, d \tau \\[6pt]
+ \frac{1}{2}\Big(u_n(0, x-t) + u_n(0, x+t) \Big) + \frac{1}{2} \int_{x-t}^{x+t} \partial_t u_n(0, y) \, dy. 
\end{multline*}
As in the proof of the uniqueness, we then have, for $1 \le p < + \infty$ and $0 < t < T$,   
\begin{multline}\label{pro-WP-p5}
\int_0^1 |u_n(t, x)|^p + |\partial_t u_n(t, x)|^p + |\partial_x u_n(t, x)|^p \, d x \\[6pt]
\le C  \int_0^t \int_0^1 \Big( |\partial_t u(s, y)|^p + |\partial_x u(s, y)|^p \Big) \, d y \, ds  \\[6pt]
+ C \left( \| u_{n}(0, \cdot)\|_{W^{1, p}}^p + \| \partial_t  u_{n}(0, \cdot)\|_{L^p}^p + \int_0^t \int_0^1 |f(s, y)|^p \,dy \, ds  \right), 
\end{multline}
and,  for $p = + \infty$, 
\begin{multline}\label{pro-WP-p6}
\|u(t, \cdot)\|_{L^\infty(0, 1)}  + \|\partial_t u(t, \cdot)\|_{L^\infty(0, 1)} + \|\partial_x u(t, \cdot) \|_{L^\infty(0, 1)} \\[6pt]
\le C t \left( \|\partial_t u(t, \cdot)\|_{L^\infty\big((0,t) \times (0, 1)\big)} + \|\partial_x u(t, \cdot) \|_{L^\infty\big((0,t) \times (0, 1)\big)}  \right) \\[6pt]
+ C \left( \| u_{n}(0, \cdot)\|_{W^{1, \infty}} + \| \partial_t  u_{n}(0, \cdot)\|_{L^\infty} + \| f\|_{L^\infty \big((0, t) \times (0, 1) \big)} \right). 
\end{multline}
Letting $n \to + \infty$, we derive \eqref{pro-WP-statement} from \eqref{pro-WP-p5} and \eqref{pro-WP-p6}. 

To derive that $u \in C([0, T]; W^{1, p}_0(0, 1)) \cap C^1([0, T]; L^p(0, 1))$ in the case $1 \le p < + \infty$ and $u \in C([0, T]; W^{1, 2}_0(0, 1)) \cap C^1([0, T]; L^2(0, 1))$ otherwise, one just notes that $(u_n)$ is a Cauchy sequence in these spaces correspondingly.  

The proof is complete. \qed

\begin{remark} \rm Our proof on the well-posedness is quite standard and is based on  the d'Alembert formula. This formula was also used previously in \cite{KMJC22}. 
\end{remark}

\begin{remark}\rm There are several ways to give the notion of weak solution even in the case $p=2$, see, e.g.,  \cite{Allaire, Coron07}.   The definitions given here is a nature modification of the case $p=2$ given in \cite{Allaire}. 
\end{remark}

\subsection{Proof of \Cref{pro-WPN}} \label{sect-proWPN} The proof of \Cref{pro-WPN} is similar to the one of \Cref{pro-WP}. To apply the d'Alembert formula, one just needs to extend various function appropriately and differently. For example, in the proof of the uniqueness, one  extend $u$ and $g$ in $(0, T) \times \mR$ by appropriate reflection  in $x$ first by even extension in $(-1, 0)$, i.e., $u(t, x) =   u (t, -x) $ and $g(t, x) =  g(t, -x)$ in $(0, T) \times (-1, 0)$
and so on.  The details are left to the reader. \qed

\section{Some useful lemmas} \label{sect-Lemmas}

In this section, we prove three lemmas which will be used through out the rest of the paper. 
The first one is quite standard and the last two ones are the main ingredients of our analysis for the Dirichlet and Neumann boundary condition. 
We begin with the following lemma.

\begin{lemma}\label{lem-Ineq} Let $1 < p < + \infty$, $0< T < \hat T_0$,
%$0< T < \hat T_0$, 
and $a \in L^\infty((0, T) \times (0, 1))$ be such that $a \ge 0$ in $(0, T) \times (0, 1)$. There exists a positive constant $C$ depending only on  $p$, $\hat T_0$, and $\| a\|_{L^\infty}$ such that, for $(\rho, \xi) \in \big[ L^p\big( (0, T) \times (0, 1) \big) \big]^2$, 
\be\label{lemH-c12}
\int_0^T \int_0^1 a |\rho - \xi|^p (t, x)  \, d x \, dt  \le  \left\{\begin{array}{cl}
C m_p & \mbox{ if } p \ge 2, \\[6pt]
C (m_p + m_p^{2/p}) & \mbox{ if } 1 < p < 2, 
\end{array}\right.
\ee
where 
\be
m_p = \int_0^1 \int_0^T a (\rho- \xi )(\rho|\rho|^{p-2} - \xi |\xi|^{p-2}) (t, x) \, dt \, dx. 
\ee
\end{lemma}  

\begin{proof} The proof of \Cref{lem-Ineq} is quite standard. For the convenience of the reader, we present its proof. There exists a positive constant $C_p$ depending only on $p$ such that 
\begin{itemize}
\item for $2 \le p < + \infty$, it holds,  for $\alpha, \beta \in \mR$,  
$$ 
(\alpha - \beta)  (\alpha |\alpha|^{p-2} - \beta |\beta|^{p-2})  \ge C_p |\alpha  -  \beta|^p; 
$$

\item for $1 < p < 2$, it holds, for $\alpha, \beta \in \mR$ \footnote{Using the symmetry between $\alpha$ and $\beta$, one can assume $\vert \alpha\vert\geq \vert \beta\vert$ and by considering  $\beta/\vert \alpha\vert$, it is enough to prove these inequalities for $\alpha=1$ and $\beta\in (-1,1)$. One finally reduces the analysis for $\beta\in (0,1)$ and even $\beta$ close to one. The conclusion follows by performing a Taylor expansion with respect to $1-\beta$.}
%This inequality can be proved by considering first the case $\alpha$ and $\beta$ have the %same size (since the other case is clear). We then consider two cases: i) $|\alpha|$ is of the %same order of  $|\beta|$, for which we obtain $(\alpha - \beta)  (\alpha |\alpha|^{p-2} - \beta |%%\beta|^{p-2})  \ge C_p |\alpha  -  \beta|^p$
 %and ii) the the remaining case where we obtain  $(\alpha - \beta)  (\alpha |\alpha|^{p-2} - \beta |\beta|^{p-2})  \ge C |\alpha  -  \beta|^p$.}, 
$$
(\alpha - \beta)  (\alpha |\alpha|^{p-2} - \beta |\beta|^{p-2})  \ge C_p \min\big\{ |\alpha  -  \beta|^p, |\alpha - \beta|^2 \big\}. 
$$
\end{itemize}
Using this, we derive that   
$$
 \mathop{\int_0^T \int_0^1}_{|\rho - \xi| \ge 1} a |\rho - \xi|^p  \, d x \, dt 
+  \mathop{\int_0^T \int_0^1}_{|\rho - \xi| < 1} a |\rho  - \xi |^{\max\{p, 2\}}  \, d x \, dt  \le m_p. 
$$
This yields 
\be \label{lemH-c1}
\int_0^T \int_0^1 a |\rho- \xi|^p (t, x)  \, d x \, dt  \le C m_p \mbox{ if } p \ge 2,  
\ee
and, using H\"older's inequality, one gets 
\be\label{lemH-c2}
\int_0^T \int_0^1 a |\rho - \xi|^p (t, x)  \, d x \, dt  \le C (m_p + m_p^{2/p}) \mbox{ if }  1< p \le 2,  
\ee
The conclusion follows from \eqref{lemH-c1} and \eqref{lemH-c2}. 
\end{proof}

The following lemma is one of the main ingredients in the analysis of the Dirichlet and Neumann boundary conditions. 

\begin{lemma} \label{lemH} Let $1 <  p < + \infty$, $0<T_0 < T < \hat T_0$,
%<   \hat T_0 $, 
$\eps_0 > 0$, $\lambda > 0$, and  $a \in L^\infty((0, T) \times (0, 1))$ be such that $T > T_0 + 4 \eps_0$, $a \ge 0$ and $a \ge \lambda > 0$ in $(0, T) \times (x_0 - \eps_0, x_0 + \eps_0) \subset (0, T) \times (0, 1)$ for some $x_0 \in (0, 1)$. Let $(\rho, \xi)$ be a broad solution of the system 
\be \label{SysH}
\left\{\ba{cl}
\rho_t - \rho_x = - \frac{1}{2} a (\rho - \xi) & \mbox{ in } (0, T) \times (0, 1), \\[6pt]
\xi_t + \xi_x = \frac{1}{2} a  (\rho - \xi) & \mbox{ in } (0, T) \times (0, 1).  
\ea\right.
\ee
Set 
\be
m_p = \int_0^1 \int_0^T a (\rho- \xi )(\rho|\rho|^{p-2} - \xi |\xi|^{p-2}) (t, x) \, dt \, dx. 
\ee
Then there exists $z \in (x_0 - \eps_0/2, x_0 + \eps_0/2)$  such that 
\begin{multline}
\int_0^{\eps_0/2} \int_0^T |\rho(t + s, z) - \rho(t, z)|^p \, dt \, ds  + \int_0^{\eps_0/2} \int_0^T |\xi(t + s, z) - \xi(t, z)|^p \, dt \, ds 
\\[6pt] 
+  \int_0^T |\rho(t, z) - \xi(t, z)|^p \, dt + \int_0^T \int_0^1 a |\rho - \xi|^p (t, x)  \, d x \, dt \\[6pt]
\le \left\{ \begin{array}{cl} Cm_p & \mbox{ if } p \ge 2, \\[6pt]
C (m_p +  m_p^{2/p}) & \mbox{ if } 1\le p< 2. 
\end{array} \right. 
\end{multline}
for some positive constant $C$ depending only on $\eps_0, \lambda, p$, $T_0$, $\hat T_0$,  and $ \| a\|_{L^\infty}$. 
\end{lemma}

\begin{proof}  Set 
$$
T_1 = T - 4 \eps_0  \quad \mbox{ and } \quad T_2 = T - 2 \eps_0. 
$$
Then $T > T_2  > T_1  > T_0$. 

We have, for $s \in (- \eps_0/2, \eps_0/2) $ and $y \in (x_0 -\eps_0/2, x_0 + \eps_0/2)$, 
\begin{multline}\label{lemH-p3}
\rho(t, y+2s)  - \rho(t, y)= \Big( \rho(t+2s, y) - \rho (t+s, y+s) \Big) + \Big( \rho (t+s, y+s) - \xi(t+s, y+s) \Big)  \\[6pt] + \Big( \xi(t+s, y + s) - \xi(t, y ) \Big) + \Big( \xi(t, y) - \rho(t, y)\Big).  
\end{multline}
By the characteristics method,   we obtain 
\be\label{lemH-p3-1}
\xi(t+s, y+s) - \xi (t, y ) = \frac{1}{2}\int_{0}^s a(t+  \tau, y + \tau ) \Big(\rho  (t+  \tau, y + \tau ) - \xi (t+  \tau, y + \tau ) \Big)\, d \tau
\ee
and 
\begin{multline}\label{lemH-p3-2}
\rho(t+2s, y) - \rho (t+s , y+s) \\[6pt]
=  \frac{1}{2}\int_{s}^{2s} a(t + \tau, y + 2s- \tau ) \Big(\rho  (t+   \tau, y + 2 s -  \tau ) - \xi (t+ \tau, y +2 s -  \tau ) \Big)\, d \tau. 
\end{multline}
Combining \eqref{lemH-p3},  \eqref{lemH-p3-1}, and \eqref{lemH-p3-2},  after integrating with respect to $t$ from 0 to $T_1$,  we obtain,  for $0 \le s \le \eps_0/2$, 
\begin{align*} 
\int_0^{T_1} |\rho(t+ 2 s, y) - \rho(t, y)|^p \, dt  \le 4^{p-1} & \left(  \int_0^{T_2} |\rho(t, y +  s) - \xi(t, y +  s)|^p  \, d t \right.\\[6pt]
& \left.  +  2 \int_0^{T_2} \int_0^1 a^p |\rho - \xi|^p (t, x) \, dt \, d x + \int_0^{T} |\rho(t, y) - \xi(t, y)|^p \, dt  \right). 
\end{align*}
Integrating the above inequality with respect to $s$ from $0$ to $\eps_0/2$, we obtain 
\begin{multline}\label{lemH-part1}
\int_{0}^{\eps_0/2}\int_0^{T_1} |\rho(t+ 2 s, y) - \rho(t, y)|^p \, dt \, ds  \\[6pt]
\le 4^p \left( \int_{x_0- \eps_0}^{x_0 + \eps_0}\int_0^{T} |\rho (t, x)- \xi (t, x)|^p \, dt \, dx \right. \\[6pt]
\left. + \eps_0 \int_0^{T} |\rho (t, y) - \xi(t, y)|^p \, dt  +  \eps_0 \int_0^1  \int_0^{T} a^p |\rho - \xi|^p (t, x) \, dt \, dx  \right). 
\end{multline}

Similarly, we have 
\begin{multline}\label{lemH-part2}
\int_{0}^{\eps_0/2}\int_0^{T_1} |\xi(t+ 2 s, y) - \xi(t, y)|^p \, dt \, ds  \\[6pt]
\le 4^p \left( \int_{x_0- \eps_0}^{x_0 + \eps_0}\int_0^{T} |\rho (t, x)- \xi (t, x)|^p \, dt \, dx \right. \\[6pt]
\left. + \eps_0 \int_0^{T} |\rho (t, y) - \xi(t, y)|^p \, dt  +  \eps_0 \int_0^1  \int_0^{T} a^p |\rho - \xi|^p (t, x) \, dt \, dx  \right). 
\end{multline}

Take $y = z \in (x_0 - \eps_0/2, x_0 + \eps_0/2)$ such that 
\be\label{lemH-part3}
\int_0^T |\rho(t, z) - \xi(t, z)|^p \, dt  \le \frac1{\eps_0} \int_{x_0 - \eps_0}^{x_0 + \eps_0} \int_0^T |\rho - \xi|^p (t, x) \, dx \, dt . 
\ee
By choosing $y=z$ in \eqref{lemH-part1} and \eqref{lemH-part2}, then by using \eqref{lemH-part3} and the fact that (itself consequence of \eqref{eq:hyp-a})
$$
 \int_{x_0- \eps_0}^{x_0 + \eps_0}\int_0^{T} |\rho (t, x)- \xi (t, x)|^p \, dt \, dx\leq C(a,p)
\int_0^1  \int_0^{T} a^p |\rho - \xi|^p (t, x) \, dt \, dx,
$$
for some positive constant $C(a,p)$ only depending on $a,p$, one gets the conclusion.  
\end{proof}

The next lemma is also a main ingredient of our analysis for the Dirichlet and Neumann boundary conditions.

\begin{lemma}\label{lemM} Let $1 \le p < + \infty$ and $L > l > 0$,  and let $u \in L^p(0, L+l)$. 
Then there exists a positive constant $C$ depending only on $p$, $L$, and  $l$ such that 
\be
\int_0^L |u(x) - \fint_0^L u(y) \, dy |^p \, dx  \le C \int_0^l \int_0^L |u(x+ s) - u(x)|^p \, dx \, d s.    
\ee
\end{lemma}

Here and in what follows, $\fint_a^b $ means  $\frac{1}{b-a} \int_a^b$ for $b> a$. 

\begin{proof} By scaling, one can assume that $L=1$. Fix $n \ge 2$ such that $2/ n \le l \le 2/ (n-1)$. 

One first notes that, for $x \in [0, 1]$, 
\begin{multline}\label{lemF-p1}
\fint_x^{x+ 1/n} \left|u(x) - \fint_x^{x+ 1/n} u(y) \, dy \right|^p \, dx \mathop{\le}^{\mathrm{Jensen}} 
 \fint_x^{x+ 1/n} \fint_x^{x+ 1/n} |u(x) -  u(y)|^p \, dx \, dy \\[6pt]
 \le n^2 \int_0^{2/n} \int_0^1 |u(x+ s) - u(x)|^p \, dx \, d s 
\end{multline}
and 
\begin{multline}\label{lemF-p2}
\left|\fint_x^{x+ 1/n} u(s) \, ds  - \fint_{x+ 1/n}^{x + 2/n} u(t) \, dt \right|^p \mathop{\le}^{\mathrm{Jensen}}  
\fint_x^{x+ 1/n} \fint_{x+ 1/n}^{x + 2/n} |u(s) - u(t)|^p \, dt  \, ds  \\[6pt]
 \le n^2 \int_0^{2/n} \int_0^1 |u(x+ s) - u(x)|^p \, dx \, d s. 
\end{multline}
For $0 \le k \le n-1$, set
$$
a_k = \fint_{k/n}^{k/n+ 1/n} u(s) \, ds. 
$$
We then derive from \eqref{lemF-p2} that, for  $0 \le i <  j \le n-1$,
\begin{multline*}
|a_j - a_i|^p \le (|a_{i+1}- a_i| + \dots + |a_j - a_{j-1}|)^p \\[6pt]
\le n^{p-1}(|a_{i+1}- a_i|^p + \dots + |a_j - a_{j-1}|^p) \\[6pt]
\le n^{p+1} \int_0^{2/n} \int_0^1 |u(x+ s) - u(x)|^p \, dx \, d s.  
\end{multline*}
This implies,  for $0 \le k \le n-1$,  
\begin{multline}\label{lemF-p3}
\left|a_k - \int_{0}^{1} u(t) \, dt \right|^p \le  \left| \frac{1}{n} \sum_{i=0}^{n-1} |a_k - a_i|
\right|^p \\[6pt]
\le \frac{1}{n} \sum_{i=0}^{n-1} |a_k - a_i|^p  \le  n^{p+1}  \int_0^{2/n} \int_0^1  |u(x+ s) - u(x)|^p \, dx \, d s. 
\end{multline}
We have 
\begin{multline}\label{lemF-p4}
\int_0^1 \left|u(x) - \int_0^1 u(y) \, dy \right|^p \, dx = \sum_{k=0}^{n-1}  \int_{k/n}^{k/n + 1/n} \left|u(x) - \int_0^1 u(y) \, dy \right|^p \, dx \\[6pt]
\le 2^{p-1}\sum_{k=0}^{n-1}  \int_{k/n}^{k/n + 1/n} \left|u(x) - a_k \right|^p \, dx + 2^{p-1}\sum_{k=0}^{n-1} \left|a_k - \int_0^1 u(y) \, dy \right|^p 
\end{multline}
The conclusion with $C = 2^{p} n^{p+1}$ now follows from \eqref{lemF-p1}, \eqref{lemF-p3}, and \eqref{lemF-p4} after noting that $L = 1$ and $2/n \le l$. 
\end{proof}

\begin{remark} \rm Related ideas used in the proof of \Cref{lemM} was implemented in the proof of Caffarelli-Kohn-Nirenberg inequality for fractional Sobolev spaces \cite{NgSq18}.
\end{remark}

\section{Exponential decay in $L^p$-framework for the Dirichlet boundary condition}  
\label{sect-Dirichlet}

In this section, we prove \Cref{thm-W} and \Cref{thm-H}. We begin with the proof \Cref{thm-H} in the first section, and then use it to prove \Cref{thm-W} in the second section. 
We finally extend these results for $a$ which might be negative in some regions using a standard perturbation argument in the third section.   

\subsection{Proof of \Cref{thm-H}} We will only consider smooth solutions $(\rho, \xi)$ \footnote{We thus assume that $a$ is smooth. Nevertheless, the constants in the estimates which will be derived in the proof depend only on $p$, $\| a\|_{L^\infty}$, $\lambda$, and $\eps_0$.}. The general case will follow by regularizing arguments.  Moreover, replacing $(\rho,\xi)$ by $(\rho-c_0,\xi-c_0)$, where the constant $c_0$ is defined in \eqref{eq:c}, we can assume that 
$$
\int_0^1(\rho_0+\xi_0)\ dx=0.
$$
Multiplying the equation of $\rho$ with $\rho |\rho|^{p-2}$, the equation of $\xi$ with $ \xi |\xi|^{p-2}$, and integrating the expressions with respect to $x$, after using the boundary conditions, we obtain, for $t > 0$,  
\be
\frac{1}{p} \frac{d}{dt} \int_0^1 (|\rho(t, x)|^p + |\xi (t, x)|^p) \, dx + \frac{1}{2} \int_0^1 a (\rho - \xi) (\rho|\rho|^{p-2} - \xi |\xi|^{p-2}) (t, x) \, dx  = 0. 
\ee
This implies 
\be \label{thmH-energy}
\frac{1}{p}\| (\rho, \xi) (t, \cdot) \|_{L^p(0, 1)}^p +  \frac{1}{2} \int_0^t \int_0^1 a  (\rho - \xi) (\rho |\rho|^{p-2} - \xi |\xi|^{p-2})(t, x) \, dx \, d t   = \frac{1}{p}  \|(\rho_0, \xi_0)\|_{L^p(0, 1)}^p.
\ee
Integrating the equations of $\rho$ and $\xi$, summing them up  and using the boundary conditions, we obtain 
$$
\frac{d}{dt}\int_0^1 \Big(\rho(t, x) + \xi(t, x)\Big) \, dx = 0 \mbox{ for } t > 0. 
$$
It follows that 
\be  \label{thmH-mass} 
\int_0^1 \Big(\rho(t, x) + \xi(t, x)\Big) \, dx = \int_0^1 \Big(\rho(0, x) + \xi(0, x)\Big) \, dx = 0 \mbox{ for } t \geq 0. 
\ee
By \eqref{thmH-energy} and \eqref{thmH-mass}, to derive \eqref{thmH-conclusion}, it suffices to prove that there exists  a constant $c > 0$ depending only on $\| a \|_{L^\infty(\mR_+ \times (0, 1))}$, $\eps_0$, $\gamma$,  and $p$ such that for any $T>2$, there exists $c_T>0$ only depending on $p,T,a$ so that 
%\footnote{It holds for $T>2$ with $c = c_T$.}, 
%\footnote{It is ok for any $T>2$.}, 
\be \label{thmH-dissipation}
\int_0^T \int_0^1 a  (\rho - \xi)  (\rho |\rho|^{p-2} - \xi |\xi|^{p-2}) (t, x) \, dx \, dt  \ge c_T \|(\rho_0, \xi_0)\|_{L^p(0, 1)}^p. 
\ee

By scaling, without loss of generality, one might assume that 
\be \label{thmH-p1}
\|(\rho_0, \xi_0)\|_{L^p(0, 1)} = 1
\ee

Set 
$$
m_p : = \int_0^T \int_0^1 a  (\rho - \xi)  (\rho |\rho|^{p-2} - \xi |\xi|^{p-2}) (t, x) \, dx \, dt. 
$$
Applying \Cref{lem-Ineq}, we have
\be\label{thmH-part11}
 \int_0^T \int_0^1 a |\rho - \xi|^p (t, x)  \, d x \, dt   \le C (m_p + m_p^{2/p}).  
\ee
By \Cref{lemH} there exists $z \in (x_0 - \eps_0/2, x_0 + \eps_0/2)$ such that 
\begin{multline}\label{thmH-part1}
\int_0^{\eps_0/2} \int_0^T |\rho(t + s, z) - \rho(s, z)|^p \, dt \, ds  + \int_0^{\eps_0/2} \int_0^T |\xi(t + s, z) - \xi(s, z)|^p \, dt \, ds 
\\[6pt] 
+  \int_0^T |\rho(t, z) - \xi(t, z)|^p \, dt \le C (m_p + m_p^{2/p}).
\end{multline}
By \Cref{lemM}, we have 
\be \label{thmH-part2}
\int_0^T |\rho (t, z) -A_\rho |^p \, dt 
\le C \int_0^{\eps_0/2} \int_0^T |\rho(t + s, z) - \rho(s, z)|^p \, dt \, ds
\ee
and 
\be\label{thmH-part3}
\int_0^T |\xi (t, z) - A_\xi |^p \, dt 
\le C \int_0^{\eps_0/2} \int_0^T |\xi(t + s, z) - \xi(s, z)|^p \, dt \, ds. 
\ee
where we have set 
\be\label{eq:cst}
A_\rho:=\fint_0^T \rho(s, z) \, ds,\quad A_\xi:=\fint_0^T \xi(s, z) \, ds.
\ee
Combining \eqref{thmH-part1}, \eqref{thmH-part2}, and \eqref{thmH-part3} yields 
\begin{multline}\label{thmH-part4}
 \int_0^{T} |\rho(t, z) -A_\rho|^p \, d t +  \int_0^{T} |\xi(t, z) -A_\xi|^p \, d t \\[6pt]  + \int_0^{T} |\rho(t, z) - \xi(t, z)|^p  \, d t
 \le C (m_p + m_p^{2/p}).  
\end{multline}
We next prove the following estimates 
\be\label{thmH-part51}
\int_{0}^1|\rho(0, x) - A_\xi|^p \, d x  \le C (m_p + m_p^{2/p}) 
\ee
and 
\be\label{thmH-part52}
\int_{0}^1|\xi(0, x) - A_\rho|^p \, d x \le C (m_p + m_p^{2/p}). 
\ee
The arguments being similar, we only provide that of \eqref{thmH-part51}. For $x\in (0,1)$, one has, by using the boundary condition at $x =0$, i.e., $\rho(\cdot,0)=\xi(\cdot,0)$,
\begin{eqnarray*}
\rho(0,x)&=&\Big(\rho(0,x)-\rho(x,0)\Big)+\rho(x,0)\\
&=&\Big(\rho(0,x)-\rho(x,0)\Big)+\xi(x,0)\\
&=&\Big(\rho(0,x)-\rho(x,0)\Big)+\Big(\xi(x,0)-\xi(x+z,z)\Big)+\xi(x+z,z),
\end{eqnarray*}
which yields, after substracting $ A_\xi$ to both sides of the above equality, 
\begin{multline}
\int_0^1\left|\rho(0,x)- A_\xi \right|^p \, dx \leq 3^{p-1}\left(\int_0^1\left|\rho(0,x)-\rho(x,0) \right|^p \, dx+\int_0^1\left|\xi(x,0)-\xi(x+z,z) \right|^p \, dx \right. \\[6pt]
\left. +\int_0^1\left|\xi(x+z,z) - A_\xi \right|^p \, dx\right). \label{eq:est-0}
\end{multline}
We use the characteristics method and \eqref{lemH-p3-1},\eqref{lemH-p3-2}
to upper bound the first two integrals in the right-hand side of \eqref{eq:est-0} by $C (m_p + m_p^{2/p})$. As for the third integral in the right-hand side of \eqref{eq:est-0},
we perform the change of variables $t = x+ z$ to obtain 
\begin{eqnarray*}
\int_0^1\left|\xi(x+z,z) - A_\xi \right|^p \, dx &=&
\int_z^{z+1}\left|\xi(t,z) - A_\xi \right|^p \, dt\\
&\leq&\int_0^T\left|\xi(t,z) - A_\xi \right|^p \, dt,
\end{eqnarray*}
which is upper bounded by $C (m_p + m_p^{2/p})$ according to \eqref{thmH-part4}. 
The proof of \eqref{thmH-part51} is complete.

We now resume the argument for \eqref{thmH-dissipation}. We start by noticing that, for every $t\in (0,T)$
$$
|A_\rho-A_\xi|\leq |A_\rho-\rho(t,z)|+|A_\xi-\rho(t,z)|+|\rho(t,z)-\xi(t,z)|.
$$
Taking the $p$-th power, integrating over $t\in (0,T)$ and using \eqref{thmH-part4}, one gets that
\be\label{eq:fin1}
|A_\rho-A_\xi|^p\leq C (m_p + m_p^{2/p}).
\ee
Similarly, for every $x\in (0,1)$,
$$
A_\rho+A_\xi=\big(A_\rho-\xi(0,x)\big)+\big(A_\xi-\rho(0,x)\big)+\big(\rho(0,x)+\xi(0,x)\big).
$$
Integrating over $x\in (0,1)$ and using \eqref{thmH-mass}, then taking the $p$-th power and using 
\eqref{thmH-part51} and \eqref{thmH-part52} yield
\be\label{eq:fin2}
|A_\rho+A_\xi|^p\leq C (m_p + m_p^{2/p}).
\ee
Still, for  $x\in (0,1)$, it holds 
$$
|\rho(0,x)|^p+|\xi(0,x)|^p\leq 2^{p-1}\Big(|A_\rho-\xi(0,A_\xi|^p+|A_\rho-\xi(0,x)|^p\Big)+
|A_\rho|^p+|A_\xi|^p.
$$
Integrating over $x\in (0,1)$ and using \eqref{thmH-p1}, one gets
\be\label{eq:fin3}
1\leq |A_\rho|^p+|A_\xi|^p+C (m_p + m_p^{2/p}).
\ee
Since it holds $|a|^p+|b|^p\leq |a+b|^p+|a-b|^p$ for every real numbers $a,b$, one deduces from 
\eqref{eq:fin1}, \eqref{eq:fin2} and \eqref{eq:fin3} that
$$
1\leq C (m_p + m_p^{2/p})
$$
and hence  $m_p \ge c_3$ for some positive constant depending only on $\| a \|_{L^\infty(\mR_+ \times (0, 1))}$, $\eps_0$, $\gamma$,  and $p$ (after fixing for instance $T=3$). 
The proof of the theorem is complete. 
\qed

%%\begin{remark} \rm Make the comments for the case $p=1$. One cannot expect the decay! \end{remark}

\subsection{Proof of \Cref{thm-W}} Using \Cref{thm-H}, we obtain the conclusion of \Cref{thm-W} for smooth solutions. The proof in the general case follows from the smooth case by density arguments. 
\qed 

\subsection{On the case $a$ not being non-negative} \label{sect-relax}

In this section, we first consider the following perturbed system of \eqref{Sys-H}:
\be \label{Sys-Hp}
\left\{\ba{cl}
\rho_t - \rho_x = - \frac{1}{2} a (\rho - \xi) - b (\rho - \xi)  & \mbox{ in } \mR_+ \times (0, 1), \\[6pt]
\xi_t + \xi_x = \frac{1}{2} a (\rho - \xi)  + b (\rho - \xi)  & \mbox{ in } \mR_+ \times (0, 1), \\[6pt]
\rho(t, 0) - \xi (t, 0) = \rho(t, 1) - \xi(t, 1) = 0 \mbox{ in } \mR_+. 
\ea\right.
\ee 
We establish the following result.

\begin{theorem}\label{thm-Hp} Let $1 < p < + \infty$, $\eps_0 > 0$, $\lambda > 0$, and $a, b \in L^\infty \big(\mR_+ \times (0, 1) \big)$ be such that $a \ge 0$ and $a \ge \lambda > 0$ in $\mR_+ \times (x_0 - \eps_0, x_0 + \eps_0) \subset \mR_+ \times (0, 1)$ for some $x_0 \in (0, 1)$.  There exists a positive constant $\alpha$ depending only on $p$,  $\| a\|_{L^\infty}$, $\eps_0$, and $\lambda$ such that if 
\be
\| b\|_{L^\infty}  \le \alpha, 
\ee
then there exist constants $C,\gamma>0$ depending only on $p$,  $\| a\|_{L^\infty}$, $\eps_0$, and $\lambda$ such that, if $\int_0^1 \rho_0 + \xi_0 \, dx =0 $,  then the solution $(\rho, \xi)$ of \eqref{Sys-Hp} satisfies 
\be \label{thmHp-conclusion}
\| (\rho, \xi) (t, \cdot) \|_{L^p(0, 1)} \le C e^{-\gamma t} \|(\rho_0, \xi_0)\|_{L^p(0, 1)}, \ t\geq 0.
\ee
\end{theorem}

\begin{proof} Multiplying the equation of $\rho$ with $\rho |\rho|^{p-2}$, the equation of $\xi$ with $ \xi |\xi|^{p-2}$, and integrating the expressions with respect to $x$, after using the boundary conditions, we obtain 
\begin{multline*}
\frac{1}{p} \frac{d}{dt} \int_0^1 (|\rho(t, x)|^p + |\xi (t, x)|^p) \, dx + \frac{1}{2} \int_0^1 a (\rho - \xi) (\rho|\rho|^{p-2} - \xi |\xi|^{p-2}) (t, x) \, dx  \\[6pt] + \int_0^1 b (\rho - \xi) (\rho|\rho|^{p-2} - \xi |\xi|^{p-2}) (t, x) \, dx = 0. 
\end{multline*}
This implies 
\begin{multline} \label{thmHp-energy}
\frac{1}{p}\| (\rho, \xi) (t, \cdot) \|_{L^p(0, 1)}^p +  \frac{1}{2} \int_0^t \int_0^1 a  (\rho - \xi) (\rho |\rho|^{p-1} - \xi |\xi|^{p-1})(t, x) \, dx \, d t   \\[6pt]
+  \int_0^t\int_0^1 b (\rho - \xi) (\rho|\rho|^{p-2} - \xi |\xi|^{p-2}) (t, x) \, dx  = \frac{1}{p}  \|(\rho_0, \xi_0)\|_{L^p(0, 1)}^p.
\end{multline}
Integrating the equation of $\rho$ and $\xi$ and using the boundary condition, we obtain 
$$
\frac{d}{dt}\int_0^1 \Big(\rho(t, x) + \xi(t, x)\Big) \, dx = 0, \mbox{ for } t > 0. 
$$
It follows that 
\be  \label{thmHp-mass}
\int_0^1 \Big(\rho(t, x) + \xi(t, x)\Big) \, dx = \int_0^1 \Big(  \rho(0, x) + \xi(0, x)  \Big) \, dx = 0, \mbox{ for } t > 0. 
\ee
By \eqref{thmHp-energy} and \eqref{thmHp-mass}, to derive \eqref{thmHp-conclusion}, it suffices to prove that there exists  a constant $c > 0$ depending only on $\| a \|_{L^\infty(\mR_+ \times (0, 1))}$, $\eps_0$, $\gamma$,  and $p$ such that for $T=3$ \footnote{It holds for $T>2$ with $c = c_T$.}, it holds 
\be \label{thmHp-dissipation}
\int_0^T \int_0^1 a  (\rho - \xi)  (\rho |\rho|^{p-2} - \xi |\xi|^{p-2}) (t, x) \, dx \, dt  \ge c \|(\rho_0, \xi_0)\|_{L^p(0, 1)}^p. 
\ee
Using the facts that $a\geq 0$ and $b$ is bounded, a simple application of  Gronwall's lemma to   \eqref{thmHp-energy} yields the existence of $\alpha>0$ depending only on $\| b \|_{L^\infty(\mR_+ \times (0, 1))}$ so that 
\be\label{thm-Hp-A}
\| (\rho, \xi) (t, \cdot) \|_{L^p(0, 1)}^p \le e^{p \alpha t} \| (\rho, \xi) (0, \cdot) \|_{L^p(0, 1)}^p \mbox{ for } t \in [0, T].
\ee
Let $(\rho_1, \xi_1)$ be the unique solution of the system 
\be \label{thm-Hp-1}
\left\{\ba{cl}
\rho_{1, t} - \rho_{1, x} = - \frac{1}{2} a (\rho_1 - \xi_1) - b  (\rho - \xi)  & \mbox{ in } \mR_+ \times (0, 1), \\[6pt]
\xi_{1, t} + \xi_{1, x} = \frac{1}{2} a (\rho_1 - \xi_1)  + b (\rho - \xi)  & \mbox{ in } \mR_+ \times (0, 1), \\[6pt]
\rho_1(t, 0) - \xi_1 (t, 0) = \rho_1(t, 1) - \xi_1(t, 1) = 0 & \mbox{ in } \mR_+, \\[6pt] 
\rho_1(0, \cdot) = \xi_1 (0, \cdot) =  0  &  \mbox{ in } (0, 1). 
\ea\right.
\ee  
Thus $- b  (\rho - \xi)$ and $b  (\rho - \xi)$ can be considered as source terms for the system of $(\rho_1, \xi_1)$. We then derive from \eqref{thm-Hp-A} that   
\be \label{thm-Hp-11}
\| (\rho_1, \xi_1) \|_{L^p(T, \cdot)} \le C \alpha \| (\rho, \xi) (0, \cdot) \|_{L^p(0, 1)}^p. 
\ee

Set 
$$
\hrho = \rho - \rho_1 \quad \mbox{ and } \quad \hxi = \xi - \xi_1. 
$$
Then 
\be \label{thm-Hp-2}
\left\{\ba{cl}
\hrho_{t} - \hrho_{x} = - \frac{1}{2} a(t, x) (\hrho - \hxi)& \mbox{ in } \mR_+ \times (0, 1), \\[6pt]
\hxi_{t} + \hxi_{x} = \frac{1}{2} a(t, x) (\hrho - \hxi) & \mbox{ in } \mR_+ \times (0, 1), \\[6pt]
\hrho(t, 0) - \hxi (t, 0) = \hrho(t, 1) - \hxi(t, 1) = 0 & \mbox{ in } \mR_+, \\[6pt] 
\hrho(0, \cdot) = \rho_0, \quad  \hxi(0, \cdot)= \xi_0 & \mbox{ in } (0, 1). 
\ea\right.
\ee  
Applying \Cref{thm-H}, we have 
\be\label{thm-Hp-B}
\| (\hrho, \hxi)(T, \cdot) \|_{L^p} \le c\| (\hrho, \hxi)(0, \cdot) \|_{L^p} 
\ee
for some positive constant $c$ depending only on $\| a\|_{L^\infty}$, $\eps_0$, and $\lambda$. 
The conclusion now follows from \eqref{thm-Hp-11} and \eqref{thm-Hp-2}. 
\end{proof}

Regarding the wave equation, we have 
\begin{theorem}\label{thm-Wp} Let $1 < p < + \infty$, $\eps_0> 0$, $\lambda > 0$, and  $a, b \in L^\infty \big(\mR_+ \times (0, 1) \big)$ be such that $a \ge 0$ and $a \ge \lambda > 0$ in $\mR_+ \times (x_0 - \eps_0, x_0 + \eps_0) \subset \mR_+ \times (0, 1)$.  There exists a positive constant $\alpha$ depending only on $p$,  $\| a\|_{L^\infty}$, $\eps_0$, and $\lambda$ such that if 
\be
\| b\|_{L^\infty}  \le \alpha, 
\ee
then there exist positive constants $C$ and $\gamma$ depending on $p$, $\| a\|_{L^\infty\big(\mR_+ \times (0, 1) \big)}$, $\eps_0$, and $\lambda$  such that for all $u_0 \in W^{1, p}_0(0, 1)$ and $u_1 \in L^p(0, 1)$,  the unique weak solution $u \in C([0, + \infty); W^{1, p}_0(0, 1)) \cap C^1([0, + \infty); L^p(0, 1))$ of 
\be \label{Sys-Wp}
\left\{\ba{cl}
\partial_{tt} u - \partial_{xx} u + \Big(a(t, x) + b(t, x) \Big) \partial_t u  = 0&  \mbox{ in } \mR_+  \times (0, 1), \\[6pt]
u(t, 0) = u(t, 1) = 0 &\mbox{ in } \mR_+, \\[6pt]
u(0, \cdot) = u_0, \quad \partial_t u (0, \cdot) = u_1& \mbox{ in } (0, 1), 
\ea \right. 
\ee
satisfies 
 \be\label{Decay-Wp-p}
\| \partial_t u(t, \cdot) \|_{L^p(0, 1)}^p + \| \partial_x u(t, \cdot) \|_{L^p(0, 1)}^p  \le C e^{-\gamma t} \Big(\| u_1 \|_{L^p(0, 1)}^p + \| \partial_x u_0 \|_{L^p(0, 1)}^p \Big), \ t\geq 0.
\ee
\end{theorem}

\begin{proof}
The proof of \Cref{thm-Wp} is similar to that of \Cref{thm-W} however instead of using \Cref{thm-H} one apply \Cref{thm-Hp}. The details are left to the reader. 
\end{proof}

\section{Exponential decay in $L^p$-framework for the Neuman boundary condition} \label{sect-Neumann}

In this section, we study the decay of the solutions of the damped wave equation equipped the Neumann boundary condition and the solutions of the corresponding hyperbolic systems. Here is the first main result of this section concerning the wave equation. 

\begin{theorem}\label{thm-WN} Let $1< p < + \infty$, $\eps_0 > 0$, $\lambda > 0$,  and let $a \in L^\infty \big(\mR_+ \times (0, 1) \big)$ be such that $a \ge 0$ and $a \ge \lambda > 0$ in $\mR_+ \times (x_0 - \eps_0, x_0 + \eps_0) \subset \mR_+ \times (0, 1)$ for some $x_0 \in (0, 1)$. There exist positive constants $C$ and $\gamma$ depending only on $p$, $\| a\|_{L^\infty\big(\mR_+ \times (0, 1) \big)}$, $\eps_0$, and $\lambda$  such that for all $u_0 \in W^{1, p}(0, 1)$ and $u_1 \in L^p(0, 1)$,  the unique weak solution $u \in C([0, + \infty); W^{1, p} (0, 1)) \cap C^1([0, + \infty); L^p(0, 1))$ of \eqref{Sys-W} and \eqref{Neumann} satisfies 
 \be\label{Decay-WN-p}
\| \partial_t u(t, \cdot) \|_{L^p(0, 1)} + \| \partial_x u(t, \cdot) \|_{L^p(0, 1)}  \le C e^{-\gamma t} \Big(\| u_1 \|_{L^p(0, 1)} + \| \partial_x u_0 \|_{L^p(0, 1)} \Big), \ t\geq 0.
\ee
\end{theorem}
 
As in the case where the Dirichlet condition is considered, we use the Riemann invariants to transform \eqref{Sys-W} with Neumann boundary condition into a hyperbolic system. 
Set 
\be \label{def-rhoxiN}
\rho(t, x) = u_x(t, x) + u_t(t, x) \quad \mbox{ and } \quad \xi(t, x) = u_x(t, x) - u_t(t, x), \mbox{ for } (t, x) \in \mR_+ \times (0, 1). 
\ee 
One can check that for smooth solutions $u$ of \eqref{Sys-W}, the pair of functions $(\rho, \xi)$ defined in \eqref{def-rhoxi} satisfies the system 
\be \label{Sys-HN}
\left\{\ba{cl}
\rho_t - \rho_x = - \frac{1}{2} a (\rho - \xi) & \mbox{ in } \mR_+ \times (0, 1), \\[6pt]
\xi_t + \xi_x = \frac{1}{2} a  (\rho - \xi) & \mbox{ in } \mR_+ \times (0, 1), \\[6pt]
\rho(t, 0) + \xi (t, 0) = \rho(t, 1) + \xi(t, 1) = 0 & \mbox{ in } \mR_+. 
\ea\right.
\ee

Concerning \eqref{Sys-HN}, we prove the following result. 

\begin{theorem}\label{thm-HN} Let $1  <  p < + \infty$, $\eps_0 > 0$, $\lambda > 0$,  and $a \in L^\infty \big(\mR_+ \times (0, 1) \big)$ be such that $a \ge 0$ and $a \ge \lambda > 0$ in $\mR_+ \times (x_0 - \eps_0, x_0 + \eps_0) \subset \mR_+ \times (0, 1)$ for some $x_0 \in (0, 1)$.  Then there exist positive constants $C,\gamma$ depending only on on $p$, $\| a\|_{L^\infty\big(\mR_+ \times (0, 1) \big)}$, $\eps_0$, and $\lambda$ such that  the unique solution $u $ of \eqref{Sys-HN} with the initial condition $\rho(0, \cdot) = \rho_0$ and $\xi(0, \cdot) = \xi_0$ satisfies 
\be \label{thmHN-conclusion}
\| (\rho, \xi) (t, \cdot) \|_{L^p(0, 1)} \le C e^{-\gamma t} \|(\rho_0, \xi_0)\|_{L^p(0, 1)}. 
\ee
\end{theorem}

The rest of this section is organized as follows. The first subsection is devoted to the proof of \Cref{thm-HN} and the second subsection is devoted to the proof of \Cref{thm-WN}.

\subsection{Proof of \Cref{thm-HN}} The argument is in the spirit of that of \Cref{thm-H}. As in there,  
we will only consider smooth solutions $(\rho, \xi)$.   Multiplying the equation of $\rho$ with $\rho |\rho|^{p-2}$, the equation of $\xi$ with $ \xi |\xi|^{p-2}$, and integrating the expressions with respect to $x$, after using the boundary conditions, we obtain, for $t > 0$,  
\be
\frac{1}{p} \frac{d}{dt} \int_0^1 (|\rho(t, x)|^p + |\xi (t, x)|^p) \, dx + \frac{1}{2} \int_0^1 a (\rho - \xi) (\rho|\rho|^{p-2} - \xi |\xi|^{p-2}) (t, x) \, dx  = 0. 
\ee
This implies 
\be \label{thmHN-energy}
\frac{1}{p}\| (\rho, \xi) (t, \cdot) \|_{L^p(0, 1)}^p +  \frac{1}{2} \int_0^t \int_0^1 a  (\rho - \xi) (\rho |\rho|^{p-1} - \xi |\xi|^{p-1})(t, x) \, dx \, d t   = \frac{1}{p}  \|(\rho_0, \xi_0)\|_{L^p(0, 1)}^p.
\ee
By \eqref{thmHN-energy}, to derive \eqref{thmHN-conclusion}, it suffices to prove that there exists  a constant $c > 0$ depending only on $\| a \|_{L^\infty(\mR_+ \times (0, 1))}$, $\eps_0$, $\gamma$,  and $p$ such that for any $T>2$, there exists $c_T>0$ only depending on $p,T,a$ so that 
%$T=3$ \footnote{It holds for $T>2$ with $c = c_T$.},  
\be \label{thmHN-dissipation}
\int_0^T \int_0^1 a  (\rho - \xi)  (\rho |\rho|^{p-2} - \xi |\xi|^{p-2}) (t, x) \, dx \, dt  \ge c_T \|(\rho_0, \xi_0)\|_{L^p(0, 1)}^p. 
\ee

By scaling, without loss of generality, one might assume that 
\be \label{thmHN-p1}
\|(\rho_0, \xi_0)\|_{L^p(0, 1)} = 1
\ee

Set 
$$
m_p : = \int_0^T \int_0^1 a  (\rho - \xi)  (\rho |\rho|^{p-2} - \xi |\xi|^{p-2}) (t, x) \, dx \, dt. 
$$
Applying \Cref{lem-Ineq}, we have
\be\label{thmHN-part11}
 \int_0^T \int_0^1 a |\rho - \xi|^p (t, x)  \, d x \, dt   \le C (m_p + m_p^{2/p}).  
\ee
By \Cref{lemH} there exists $z \in (x_0 - \eps_0/2, x_0 + \eps_0/2)$ such that 
\begin{multline}\label{thmHN-part1}
\int_0^{\eps_0/2} \int_0^T |\rho(t + s, z) - \rho(t, z)|^p \, dt \, ds  + \int_0^{\eps_0/2} \int_0^T |\xi(t + s, z) - \xi(t, z)|^p \, dt \, ds 
\\[6pt] 
+  \int_0^T |\rho(t, z) - \xi(t, z)|^p \, dt \le C (m_p + m_p^{2/p}). 
\end{multline}
Applying \Cref{lemM}, we obtain 
\be \label{thmHN-part2}
 \int_0^T |\rho (t, z) - \fint_0^T \rho(s, z) \, ds |^p \, dt 
\le C \int_0^{\eps_0/2} \int_0^T |\rho(t + s, z) - \rho(t, z)|^p \, dt \, ds
\ee
and 
\be\label{thmHN-part3}
 \int_0^T |\xi (t, z) - \fint_0^T \xi(s, z) \, ds |^p \, dt 
\le C \int_0^{\eps_0/2} \int_0^T |\xi(t + s, z) - \xi(t, z)|^p \, dt \, ds. 
\ee
Combining \eqref{thmHN-part1}, \eqref{thmHN-part2}, and \eqref{thmHN-part3} yields 
\begin{multline}\label{thmHN-part4}
 \int_0^{T} |\rho(t, z) - \fint_0^{T}\rho(\tau, z) \, d \tau|^p \, d t +  \int_0^{T} |\xi(t, z) - \fint_0^{T}\xi(\tau, z) \, d \tau|^p \, d t \\[6pt]  + \int_0^{T} |\rho(t, z) - \xi(t, z)|^p  \, d t
 \le C (m_p + m_p^{2/p}).  
\end{multline}

Using the characteristics method to estimate $\rho(\tau, 0)$ by $\rho(\tau-z, z)$ and $\xi(\tau, 0)$ by $\xi(\tau+z, z)$ after using the boundary condition at $0$ and choosing appropriately $\tau$, we derive from \eqref{thmHN-part11} that 
\eqref{thmHN-part4} that 
\be\label{thmHN-part5}
\left|\fint_0^T\rho(t, z)\, d t  +  \fint_0^T \xi(t, z) d t \right|^p \le C (m_p + m_p^{2/p}). 
\ee
As done to obtain \eqref{thmH-part51} and \eqref{thmH-part52}, we use 
the characteristic methods to estimate $\rho(0, \cdot)$ via $\xi(t, z)$ and $\xi(0, \cdot)$ via $\rho(t, z)$ after taking into account the boundary conditions (at $x =0$ for $\rho(0, \cdot)$ and at $x=1$ for $\xi(0, \cdot)$), we derive from \eqref{thmHN-part11} and \eqref{thmHN-part4} that  
\be\label{thmHN-part6}
 \left|\fint_0^T\rho(t, z)\, d t \right|^p +  \left| \fint_0^T \xi(t,z) d t \right|^p \ge 1-  C (m_p + m_p^{2/p}). 
\ee
Combining \eqref{thmHN-part5} and \eqref{thmHN-part6}, we derive (after choosing $T=3$) that there exists a postive constant $c_3$ only depending on $\| a \|_{L^\infty(\mR_+ \times (0, 1))}$, $\eps_0$, $\gamma$,  and $p$ such that $m_p \ge c$. The proof of the theorem is complete. 
\qed

\subsection{Proof of \Cref{thm-WN}} The proof of \Cref{thm-WN} is in the same spirit of \Cref{thm-W}. However, instead of using \Cref{thm-H}, we apply \Cref{thm-HN}. In fact, as in the proof of \Cref{thm-W}, we have 
\begin{multline*}
\int_0^1 |\partial_t u(t, x) - \partial_x u(t, x)|^p + |\partial_t u(t, x) + \partial_x u(t, x)|^p \, dx \\[6pt]
\le C e^{-\gamma t} \int_0^1 |\partial_t u(0, x) - \partial_x u(0, x)|^p + |\partial_t u(0, x) + \partial_x u(0, x)|^p \, dx. 
\end{multline*}
Assertion \eqref{Decay-WN-p} follows with two different appropriate positive constants $C$ and $\gamma$. \qed

\begin{remark} \rm  We can also consider the setting similar to the one in \Cref{sect-relax} and establish similar results. This allows one to deal with a class of $a$ for which $a$ is not necessary to be non-negative. The analysis for this is almost the same lines as in \Cref{sect-relax} and is not pursued here. 
\end{remark}

\section{Exponential decay in $L^p$-framework for the dynamic boundary condition}
\label{sect-Dynamic}

In this section, we study the decay of the solution of the damped wave equation equipped the dynamic boundary condition and of the solutions of the corresponding  hyperbolic systems. Here is the first main result of this section concerning the wave equation. 

\begin{theorem}\label{thm-WD} Let $1 <   p <  + \infty$, $\kappa > 0$, and $a \in L^\infty \big(\mR_+ \times (0, 1) \big)$ non negative.  Then there exist positive constants $C,\gamma$ depending only on $p$, $\kappa$, and $\| a\|_{L^\infty\big(\mR_+ \times (0, 1) \big)}$ such that for all $u_0 \in W^{1, p}(0, 1)$ and $u_1 \in L^p(0, 1)$, there exists a unique weak solution $u \in C([0, + \infty); W^{1, p}(0, 1)) \cap C^1([0, + \infty); L^p(0, 1))$ such that $\partial_t u, \partial_x u  \in C([0, 1]; L^p(0, T))$ for all $T > 0$ of
\be \label{Sys-WD}
\left\{\ba{cl}
\partial_{tt} u - \partial_{xx} u + a \partial_t u = 0 & \mbox{ in } \mR_+ \times (0, 1), \\[6pt]
\partial_x u(t, 0) - \kappa \partial_t u (t, 0) = \partial_x u(t, 1) + \kappa \partial_t u(t, 1) = 0 & \mbox{ in }  \mR_+, \\[6pt]
u(0, \cdot) = u_0, \quad \partial_t u (0, \cdot) = u_1 & \mbox{ in } (0, 1), 
\ea \right. 
\ee
satisfies
 \be\label{Decay-WD-p}
\| \partial_t u(t, \cdot) \|_{L^p(0, 1)} + \| \partial_x u(t, \cdot) \|_{L^p(0, 1)}  \le C e^{-\gamma t} \Big(\| u_1 \|_{L^p(0, 1)} + \| \partial_x u_0 \|_{L^p(0, 1)} \Big),\ t\geq 0.
\ee
\end{theorem}

\begin{remark} \rm  In \Cref{thm-WD}, a weak considered  solution of \eqref{Sys-WD} means that $\partial_{tt} u (t, x) - \partial_{xx} u (t, x) + a(t, x) \partial_t u = 0$ holds in the distributional sense, and  the boundary and the initial conditions are understood as usual thanks to the regularity imposing condition on  the solutions.  
\end{remark}

As previously, we use the Riemann invariants to transform the wave equation into a hyperbolic system. 
Set 
\be \label{def-rhoxiD}
\rho(t, x) = u_x(t, x) + u_t(t, x) \quad \mbox{ and } \quad \xi(t, x) = u_x(t, x) - u_t(t, x) \mbox{ for } (t, x) \in \mR_+ \times (0, 1). 
\ee 
One can check that for smooth solutions $u$ of \eqref{Sys-W}, the pair of functions $(\rho, \xi)$ defined in \eqref{def-rhoxi} satisfies the system 
\be \label{Sys-HD}
\left\{\ba{cl}
\rho_t - \rho_x = - \frac{1}{2} a(t, x) (\rho - \xi) & \mbox{ in } \mR_+ \times (0, 1), \\[6pt]
\xi_t + \xi_x = \frac{1}{2} a(t, x) (\rho - \xi) & \mbox{ in } \mR_+ \times (0, 1), \\[6pt]
\xi (t, 0) = c_0 \rho(t, 0), \quad \rho(t, 1) = c_1 \xi(t, 1) & \mbox{ in } \mR_+,  
\ea\right.
\ee
where $c_0 = c_1 = (\kappa - 1)/ (\kappa + 1)$. 

\medskip 
Regarding System \eqref{Sys-HD} with $c_0,c_1$ not necessarily equal, we prove the following result. 

\begin{theorem}\label{thm-HD} Let $1 <  p < + \infty$, $c_0, c_1 \in (-1,1)$, and $a \in L^\infty \big(\mR_+ \times (0, 1) \big)$ non negative.  Then there exist positive constants $C,\gamma$ depending only on  $c_0$, $c_1$, and  $\| a\|_{L^\infty\big(\mR_+ \times (0, 1) \big)}$ such that  the unique solution $u $ of \eqref{Sys-HD} with the initial condition $\rho(0, \cdot) = \rho_0$ and $\xi(0, \cdot) = \xi_0$ satisfies 
\be \label{thmHD-conclusion}
\| (\rho, \xi) (t, \cdot) \|_{L^p(0, 1)} \le C e^{-\gamma t} \|(\rho_0, \xi_0)\|_{L^p(0, 1)},\ t\geq 0.
\ee
\end{theorem}

The rest of this section is organized as follows. The proof of \Cref{thm-HD} is given in the first section and the proof of \Cref{thm-WD} is given in the second section.

\subsection{Proof of \Cref{thm-HD}} We will only consider smooth solutions $(\rho, \xi)$. 
Multiplying the equation of $\rho$ with $\rho$, the equation of $\xi$ with $ \xi $, and integrating the expressions with respect to $x$, after using the boundary conditions, we obtain, for $t > 0$,  
\begin{multline}
\frac{1}{p} \frac{d}{dt} \int_0^1 (|\rho(t, x)|^p + |\xi (t, x)|^p) \, dx + \frac{1}{2} \int_0^1 a (\rho - \xi)(\rho |\rho|^{p-2} - \xi |\xi|^{p-2}) (t, x) \, dx \\[6pt]
 \frac{1}{p}  \Big( (1- |c_1|^p) |\xi(t, 1)|^p + (1- |c_0|^p) |\rho(t, 0)|^p \Big)= 0. 
\end{multline}
This implies 
\begin{multline} \label{thmHD-energy}
\frac{1}{p}\| (\rho, \xi) (t, \cdot) \|_{L^p(0, 1)}^p +  \frac{1}{2} \int_0^T \int_0^1 a (\rho - \xi)(\rho |\rho|^{p-2} - \xi |\xi|^{p-2}) (t, x) \, dx \, d t  \\[6pt]
+ \frac{1}{p} \int_0^T \Big( (1- |c_1|^p) |\xi(t, 1)|^p + (1- |c_0|^p) |\rho(t, 0)|^p \Big) \, dt  = \frac{1}{2}  \|(\rho_0, \xi_0)\|_{L^2(0, 1)}^2.
\end{multline}
To derive \eqref{thmHD-conclusion} from \eqref{thmHD-energy}, it suffices  to prove that there exists  a constant $c > 0$ depending only on $\| a \|_{L^\infty(\mR_+ \times (0, 1))}$, $c_0$, $c_1$,  $\eps_0$, $\gamma$,  and $p$ such that for for $T=3$ \footnote{It holds for $T>2$ with $c = c_T$.}, it holds 
%\footnote{It is ok for any $T>2$.}, it holds 
\begin{multline} \label{thmHD-dissipation}
\int_0^T \int_0^1 a (\rho - \xi)(\rho |\rho|^{p-2} - \xi |\xi|^{p-2}) (t, x) \, dx \, d t \\[6pt]
+ \int_0^T \Big(  |\xi(t, 1)|^p +  |\rho(t, 0)|^p \Big) \, dt   \ge c \|(\rho_0, \xi_0)\|_{L^p(0, 1)}^p. 
\end{multline}

After scaling, one might assume without loss of generality that 
\be \label{thmHD-p1}
\|(\rho_0, \xi_0)\|_{L^p(0, 1)} = 1
\ee
Applying  \Cref{lem-Ineq}, we have 
\be\label{thmHD-part11}
 \int_0^T \int_0^1 a |\rho - \xi|^p (t, x)  \, d x \, dt   \le C (m_p + m_p^{2/p}), 
\ee
where 
$$
m_p : = \int_0^T \int_0^1 a  (\rho - \xi)  (\rho |\rho|^{p-2} - \xi |\xi|^{p-2}) (t, x) \, dx \, dt. 
$$
Using the characteristics method (in particular equations \eqref{lemH-p3-1}, \eqref{lemH-p3-2}), we derive that 
\be\label{thmHD-p2}
\| (\rho, \xi) (T, \cdot) \|_{L^p(0, 1)}^p \le C \int_0^T \Big(  |\xi(t, 1)|^p +  |\rho(t, 0)|^p \Big) \, dt + C \int_0^T \int_0^1 a^p |\rho - \xi|^p (t, x)  \, d x \, dt . 
\ee
As a consequence of \eqref{thmHD-energy}, \eqref{thmHD-p1}, \eqref{thmHD-part11}, and \eqref{thmHD-p2}, we have 
$$
\int_0^T \int_0^1 a (\rho - \xi)(\rho |\rho|^{p-2} - \xi |\xi|^{p-2}) (t, x) \, dx \, d t \\[6pt]
+ \int_0^T \Big(  |\xi(t, 1)|^p +  |\rho(t, 0)|^p \Big) \, dt   \ge c. 
$$
The proof of the theorem is complete. 
\qed

\begin{remark}\label{rem-Dyn} \rm In the case $a \equiv 0$, one can show that the exponential stability for $1 \le p \le + \infty$ by noting that 
$$
\| \big(\rho(t+1, 0), \rho(t+1, 1) \big) \| \le \max\{|c_0|, |c_1| \} \| \big(\rho(t, 0), \rho(t, 1) \big) \|. 
$$ 
The conclusion then follows using the characteristics method. 
\end{remark}

\subsection{Proof of \Cref{thm-WD}} We first deal with the well-posedness of the system. The uniqueness follows as in the proof of \Cref{pro-WP} via the d'Alembert formula. The existence can be proved by approximation arguments. First deal with smooth solutions (with smooth $a$) using \Cref{thm-HD} and then pass to the limit. The details are omitted.

The proof of \eqref{thmHD-conclusion} is in the same spirit of \eqref{Decay-W-p}. However, instead of using \Cref{thm-H}, we apply \Cref{thm-HD}. The details are left to the reader. \qed

\begin{remark} \label{rem-Dynamic} \rm One can prove the well-posedness of \eqref{Sys-W} and \eqref{Dynamic} directly in $L^p$-framework. Nevertheless, to make the sense for the boundary condition, one needs to consider regular solutions and then $a$ is required to be more regular than just $L^\infty$. We here take advantage of the fact that such a system has a hyperbolic structure as given in \eqref{Sys-HD}. This give us the way to give sense for the solution by imposing the fact $\partial_t u, \partial_x u  \in C([0, 1]; L^p(0, T))$ for all $T > 0$. 
\end{remark}

\providecommand{\bysame}{\leavevmode\hbox to3em{\hrulefill}\thinspace}
\providecommand{\MR}{\relax\ifhmode\unskip\space\fi MR }
\providecommand{\MRhref}[2]{%
  \href{http://www.ams.org/mathscinet-getitem?mr=#1}{#2}
}
\providecommand{\href}[2]{#2}

\end{document}